\documentclass[11pt,reqno]{amsart}
\usepackage{diagrams}
\usepackage{amssymb}
\usepackage{enumerate}
\usepackage[pagewise]{lineno}
\numberwithin{equation}{section}

\theoremstyle{definition}
\newtheorem{thm}{Theorem}[section]
\newtheorem{cor}[thm]{Corollary}
\newtheorem{lem}[thm]{Lemma}
\newtheorem{exa}[thm]{Example}

\newtheorem{defi}[thm]{Definition}
\newtheorem{rem}[thm]{Remark}
\newtheorem{note}[thm]{Notation}

\DeclareMathOperator{\p3}{\mathbb{P}^3}

\DeclareMathOperator{\Spec}{\mathrm{Spec}}

\DeclareMathOperator{\mo}{\mathcal{O}}

\newcommand{\mr}[1]{\mathrm{#1}}
\newcommand{\mb}[1]{\mathbb{#1}}
\newcommand{\mc}[1]{\mathcal{#1}}
\newcommand{\ov}[1]{\overline{#1}}

\begin{document}

\title{Examples of varieties with index one on $C_1$-fields}

\author[A. Dan]{Ananyo Dan}

\address{BCAM - Basque Centre for Applied Mathematics, Alameda de Mazarredo 14,
48009 Bilbao, Spain}

\email{adan@bcamath.org}

\author[I. Kaur]{Inder Kaur}

\address{Instituto de Matem\'{a}tica Pura e Aplicada, Estr. Dona Castorina, 110 - Jardim Bot\^{a}nico, Rio de Janeiro - RJ, 22460-320, Brazil}

\email{inder@impa.br}

\subjclass[2010]{Primary $12$G$05$, $16$K$50$, $14$D$20$, $14$J$60$, Secondary $14$L$24$, $14$D$22$}


\date{\today}

\begin{abstract}
Let $K$ be the fraction field of a Henselian discrete valuation ring with algebraically closed residue field $k$.
In this article we give a sufficient criterion for a projective variety over such a field to have index $1$.  
\end{abstract}

\maketitle
%\linenumbers
% Begin numeric (1,2,3...) page numbering

\section{Introduction}
A field $K$ is called $C_1$ if any degree $d$ polynomial in $n$ variables with $n>d$ has a non-trivial solution.
The $C_{1}$ conjecture due to Lang, Manin and Koll\'{a}r states that every separably rationally connected variety over a $C_{1}$
field has a rational point. The conjecture has already been proven for several $C_1$-fields (see \cite{ind} for a complete discussion). 
However it is still open in the case when $K$ is the fraction field of a Henselian discrete valuation ring of characteristic $0$ with algebraically 
closed residue field of characteristic $p>0$. Recently, the conjecture was shown to hold trivially for certain rationally connected varieties 
over such fields (see \cite{ind2}).

\vspace{0.2 cm}
It is natural to ask whether a similar conjecture holds if we replace the condition for a rational point by the condition of index one
and weaken the condition on rational connectedness.
Recall that the \emph{index} of a variety $X$, denoted $\mr{ind}(X)$, is the gcd of the set of degrees of zero dimensional cycles on $X$.
In \cite[Corollary $2.5$]{esn2015}, Esnault, Levine and Wittenberg prove that if $X$ is a smooth, projective variety over the fraction field of a Henselian 
discrete valuation ring with algebraically closed residue field of characteristic $0$, then $\mr{ind}(X)$ 
divides the Euler characteristic of the structure sheaf of $X$.
Using this they prove that, in the case $X$ is a rationally connected variety over such a field, we have $\mr{ind}(X)=1$ (see \cite[Corollary $3$]{esn2015}).
Since the Euler characteristic of $\mo_X$ is one if and only if $X$ has arithmetic genus $0$, this gives a positive answer to the 
modified conjecture only for very few choices of $X$.

\vspace{0.2 cm}
In this article we study a weaker notion of index, which we call the linear index. 
Let $K$ be a field of characteristic zero and $X$ a projective $K$-variety.
 We define the \emph{linear index} of $X$, denoted $\mr{ind}_{\mr{lin}}(X)$,
to be the gcd of the Euler characteristic of the set of
 line bundles on $X$. 
The definition of linear index is inspired by Koll\'{a}r's definition of elw-index, denoted $\mr{elw}(X)$, as given in \cite{koll}, which is 
the gcd of the Euler characteristic of all coherent sheaves on $X$. An advantage of using $\mr{ind}_{\mr{lin}}(X)$ over $\mr{elw}(X)$ is that it is much easier
to compute the index of $X$ using the former notion. In particular, it is extremely hard to enumerate the set of all coherent sheaves 
on a variety (even fixing Hilbert polynomial is not sufficient to guarantee boundedness of families of coherent sheaves, see \cite{huy}).
In comparison, the set of all invertible sheaves is given by the Picard group, which is of finite rank in numerous examples.
 Moreover, the Euler characteristic of an invertible sheaf is significantly easier to compute than that of a general coherent sheaf (see 
 Riemann Roch formula for coherent sheaves \cite{baum}). As a result, we are able to give a simple combinatorial criterion under which 
 a variety has index $1$.
 More precisely, 

 \begin{thm}[Theorem \ref{elw9}, Corollary \ref{elw6}]\label{thm1}
 Fix an ample line bundle $H$ on $X$ such that if there exists an invertible sheaf $H_0$ on $X$ with $H_0^{\otimes n} \cong H$ then 
 $n$ must be $1$ or $-1$.
 Suppose that $H^1(\mathcal{O}_X)=0$, $\mr{Pic}(X_{\ov{K}})$ is of rank $r$, generated by $\mc{L}_1, ..., \mc{L}_{r-1}$ and  $\mc{L}_r:=H_{\ov{K}}=H \otimes_K \ov{K}$
   satisfying the following conditions:
   \begin{enumerate}
    \item the ideal $(\deg(\mc{L}_1),\deg(\mc{L}_2),...,\deg(\mc{L}_r))$ in $\mb{Z}$ generated by $\deg(\mc{L}_i)$ for $i=1,...,r$ coincides with the ideal $(1)$,
    where degree of the invertible sheaves are taken with respect to $H_{\ov{K}}$ (see \cite[Definition $1.2.11$]{huy}),
    \item for any $r \times r$-matrix $A=(a_{i,j})$ with integral entries $a_{i,j}$, $a_{r,k}=0$ for all $k<r$, $a_{r,r}=1$, $A \not=\mr{Id}$ and $A^t=\mr{Id}$ for some $t>0$, 
    we have $\sum_j a_{ij} \deg(\mc{L}_j)\not= \deg(\mc{L}_i)$ for some $i>0$. 
   \end{enumerate}
  Then, each $\mc{L}_i$ is $G$-invariant and $\mr{gcd}\{\chi(\mc{L}_i(n))| i=1,...,r \mbox{ and } n \in \mb{Z}\}=1$. Moreover, if $K$ is a $C_1$-field, then 
  $\mr{ind}_{\mr{lin}}(X)=\mr{ind}(X)=1 \, \mbox{ if } \mr{char}(k)=0 \mbox{ and }$ prime-to-p part of  $\mr{ind}(X)$ and  $\mr{ind}_{\mr{lin}}(X)$ equals  $1$  if  $\mr{char}(k)=p>0.$
   \end{thm}
By prime-to-p part of $N$ we mean the largest divisor of $N$ which is prime to $p$. 
   
We use the criterion in Theorem \ref{thm1} to give examples of 
non-rationally connected varieties having index one over a $C_1$-field (see Example \ref{elw8} and Remark \ref{elw11}). 
In Example \ref{elw02}, we give examples in the case $K$ is not a $C_1$-field. 
 
\vspace{0.2 cm}
\emph{Acknowledgements}
The first author is supported by ERCEA Consolidator Grant $615655$-NMST and also
by the Basque Government through the BERC $2014-2017$ program and by Spanish
Ministry of Economy and Competitiveness MINECO: BCAM Severo Ochoa
excellence accreditation SEV-$2013-0323$. The second author is funded by a fellowship from CNPq Brazil.
 A part of this work was 
done when she was visiting ICTP. She warmly thanks ICTP, the Simons Associateship and Prof. 
C. Araujo for making this possible. We also thank the referee for several helpful suggestions.
 
\section{Index of varieties}
  
 \begin{note}\label{e1}
Let $K$ be a field of characteristic $0$ and
$X$ be a projective $K$-variety. 
\end{note}
  
  \begin{defi}
   Given a smooth, quasi-projective $K$-variety $Y$, we define the associated (cohomological) \emph{Brauer group}
 $\mr{Br}(Y)=H^2_{\mbox{\'{e}t}}(Y, \mb{G}_m)$. 
  \end{defi}

  \begin{rem}
    If $K$ is the maximal unramified extension of a complete field, then
  $K$ is a $C_1$-field (see \cite{lang1}). Recall, for any $C_1$-field $K$, we have $\mr{Br}(K)=0$ (see \cite[\S X.$7$]{ser}).
  \end{rem}

   \begin{rem}
 One can check the following elementary properties of $\mr{ind}_{\mr{lin}}(X)$:
 \begin{enumerate}
  \item If $X$ is a projective $K$-curve containing a $K$-rational point, then $\mr{ind}_{\mr{lin}}(X)=1$.
  \item This is not true in higher dimension. If $K=\mb{C}$ and $X$ is a very general smooth, projective quartic surface in $\p3$ 
  (by Noether-Lefschetz theorem, a very general quartic has Picard rank one), then $\mr{ind}_{\mr{lin}}(X)$ is divisible by $2$ (use Riemann-Roch theorem).
  \item For any $X$, $\mr{ind}_{\mr{lin}}(X)$ divides the gcd of the set of Euler characteristics of $H^{\otimes a}$ as $a$
  varies over $\mb{Z}$. In particular, if $X$ is an odd degree surface in $\mb{P}^3_K$, then $\mr{ind}_{\mr{lin}}(X)=1$.
    \end{enumerate}
 \end{rem}

 \begin{lem}\label{elw01}
  Suppose $K$ is the quotient field of a Henselian discrete valuation ring $R$ with algebraically closed residue field $k$. We then have
\begin{enumerate}
 \item if $\mr{char}(k)=0$, then $\mr{ind}(X)$ divides $\mr{ind}_{\mr{lin}}(X)$,
 \item if $\mr{char}(k)=p>0$, then the prime-to-$p$ part of $\mr{ind}(X)$ divides that of $\mr{ind}_{\mr{lin}}(X)$,
%  \item for any morphism $f:X \to Y$, $\mr{ind}(Y)$ (resp. prime-to-p part of $\mr{ind}(Y)$)
%  divides $\mr{ind}_{\mr{lin}}(X)$ (resp. prime-to-p part of $\mr{ind}_{\mr{lin}}(X)$) if $\mr{char}(k)=0$ (resp. $\mr{char}(k)=p>0$).
 \end{enumerate}
 \end{lem}
 
 \begin{proof}
 The proof follows easily from \cite[Theorem $3.2$]{esn2015}.
%  Moreover, given a morphism $f:X \to Y$, the pushforward of zero-cycles preserve degree.
%  Then, $(3)$ follows directly from $(1)$ and $(2)$.
%     This proves the lemma.
 \end{proof}

 \begin{defi}
 Denote by $G$ the absolute Galois group $\mr{Gal}(\ov{K}/K)$. An invertible sheaf $\mc{L}_{\ov{K}}$ on $X_{\ov{K}}:=X \times_K \Spec(\ov{K})$ 
is called $G$-\emph{invariant} if for any $\sigma \in G$ and the induced morphism $\sigma:X_{\ov{K}} \to X_{\ov{K}}$, 
we have $\sigma^*\mc{L}_{\ov{K}} \cong \mc{L}_{\ov{K}}$.
 \end{defi}

 \begin{thm}\label{elw9}
 Let $H$ be an ample divisor on $X$ such that if there exists an invertible sheaf $H_0$ on $X$ with $H_0^{\otimes n} \cong H$, then $n=1$ or $-1$.
 Suppose that $H^1(\mathcal{O}_X)=0$, $\mr{Pic}(X_{\ov{K}})$ is of rank $r$, generated by $\mc{L}_1, ..., \mc{L}_{r-1}$ and  $\mc{L}_r:=H_{\ov{K}}=H \otimes_K \ov{K}$
   satisfying the following conditions:
   \begin{enumerate}
    \item the ideal $(\deg(\mc{L}_1),\deg(\mc{L}_2),...,\deg(\mc{L}_r))$ in $\mb{Z}$ generated by $\deg(\mc{L}_i)$ for $i=1,...,r$ coincides with the ideal $(1)$,
    where $\deg(\mc{L}_i)$ is with respect to $H_{\ov{K}}$,
    \item for any $r \times r$-matrix $A=(a_{i,j})$ with integral entries $a_{i,j}$, $a_{r,k}=0$ for all $k<r$, $a_{r,r}=1$, $A \not=\mr{Id}$ and $A^t=\mr{Id}$ for some $t>0$, 
    we have $\sum\limits_j a_{ij} \deg(\mc{L}_j)\not= \deg(\mc{L}_i)$ for some $i>0$. 
   \end{enumerate}
  Then, each $\mc{L}_i$ is $G$-invariant and $\mr{gcd}\{\chi(\mc{L}_i(n))| i=1,...,r \mbox{ and } n \in \mb{Z}\}=1$, where 
  $\mc{L}_i(n):=\mc{L}_i \otimes H_{_{\ov{K}}}^{\otimes n}$.
   \end{thm}

   \begin{proof} 
By \cite[Th\'{e}or\`{e}me $8.5.2$]{ega43}, there exists a 
finite field extension $K'$ of $K$ and an invertible sheaf $\mc{L}'_i$ on $X_{K'}=X \times_K \Spec(K')$ such that 
$\mc{L}_i \cong \mc{L}'_i \otimes_{K'} \ov{K}$
i.e., the pull-back of $\mc{L}'_i$ to $X_{\ov{K}}$ is isomorphic to $\mc{L}_i$. 
Without loss of generality (replace $K'$ by the smallest Galois extension of $K$ containing $K'$), we can assume that $K'$ is a 
finite Galois extension of $K$. Denote by $\mc{L}'_r:=H_{K'}=H \otimes_K K'$. 
Let $\sigma \in \mr{Gal}(K'/K)$ and $\sigma:X_{K'} \to X_{K'}$ the induced morphism. 
Suppose that \[\sigma^*\mc{L}'_i \cong \bigotimes\limits_{j=1}^r (\mc{L}'_j)^{\otimes a_{i,j}} \, \mbox{ for some integer } a_{i,j}, i<r.\]
As $H$ comes from $X$, we have $\sigma^* \mc{L}'_r \cong \mc{L}'_r$. 
Suppose that $\sigma^*\mc{L}'_i \not\cong \mc{L}'_i$ for some $i>0$. Then there exists $j \not= i$ such that $a_{i,j} \not= 0$.
In particular, the matrix $A:=(a_{i,j})$ is not the identity matrix. Note that, $a_{r,j}=0$ for all $j<r$ and $a_{r,r}=1$.
Since $\sigma$ is of finite order, there exists an integer $b$ such that 
\[\mc{L}'_i \cong (\sigma^*)^b (\mc{L}'_i)= \bigotimes\limits_{j=1}^r (\mc{L}'_j)^{\otimes b_{i,j}} \, \mbox{ where } A^b=(b_{i,j}), i=1,...,r, \, j=1,...,r.\]
In other words, $A^b=\mr{Id}$.
Since the Hilbert function of $\mc{L}'_i$ is the same as that of $\sigma^*\mc{L}'_i$, we conclude that 
$\deg(\mc{L}_i')=\sum_j a_{i,j} \deg(\mc{L}'_j)$. But, this contradicts our assumption $(2)$.
Hence, $\sigma^*\mc{L}'_i \cong \mc{L}'_i$ for all $i=1,...,r$. 
In other words, each $\mc{L}_i$ is $G$-invariant.

We now prove that $\mr{gcd}\{\chi(\mc{L}_i(n))|i=1,...,r \mbox{ and } n \in \mb{Z}\}=1$. 
Denote by $P_i(t)$ (resp. $P_0(t)$) the Hilbert polynomial of the invertible sheaf $\mc{L}_i$ (resp. $\mo_{X_{\ov{K}}}$) for $i=1,...,r$.
Note that the leading coefficient of $Q_i(t):=P_i(t)-P_0(t)$ is $\deg(\mc{L}_i)/(d-1)!$, where $d=\dim X$ (see \cite[Definition $1.2.11$]{huy}).
We claim that $\mr{gcd}\{Q_i(n)| n \in \mb{Z}\}$ divides $\deg(\mc{L}_i)$ for each $i=1,...,r$. 
Indeed, denote by $D^1Q_i(t):=Q_i(t+1)-Q_i(t)$ and recursively, $D^jQ_i(t):=D^{j-1}Q_i(t+1)-D^{j-1}Q_i(t)$.
Note that, $D^jQ_i(t)$ is of degree $d-1-j$ with leading coefficient $(d-1)(d-2)...(d-j) \deg(\mc{L}_i)/(d-1)!$ for all $j \ge 1$. Thus, \[D^{d-1}Q_i(t)=(d-1)!\deg(\mc{L}_i)/(d-1)!=\deg(\mc{L}_i).\]
It follows immediately, $\mr{gcd}\{Q_i(n)| n \in \mb{Z}\}$ divides $\deg(\mc{L}_i)$. This proves the claim. 
Now,
\[\mr{gcd}\{\chi(\mc{L}_i(n))|i=1,...,r \mbox{ and } n \in \mb{Z}\}\, \mbox{ divides } \mr{gcd}\{\chi(\mc{L}_i(n))-\chi(\mo_{X_{\ov{K}}}(n)) | i=1,...,r \mbox{ and } n \in \mb{Z}\}\]
which is equal to $\mr{gcd}\{Q_i(n)|n \in \mb{Z}, \, i=1,...,r\}.$
Since the $\mr{gcd}\{Q_i(n)|n \in \mb{Z}, \, i=1,...,r\}$ divides the generator of the ideal $(\deg(\mc{L}_1),\deg(\mc{L}_2), ..., \deg(\mc{L}_r))=(1)$, we conclude that 
\[\mr{gcd}\{\chi(\mc{L}_i(n))|i=1,...,r \mbox{ and } n \in \mb{Z}\} = 1.\]
This proves the theorem.
\end{proof}

 \begin{cor}\label{elw6}
  Suppose $K$ is the quotient field of a Henselian discrete valuation ring $R$ with algebraically closed residue field $k$. 
  If $X$ satisfies the hypothesis of Theorem \ref{elw9}, then 
   $\mr{ind}(X)=1  \mbox{ if } \mr{char}(k)=0 \mbox{ and } \mbox{ prime-to-p part of } \mr{ind}(X) \mbox{ equals } 1 \mbox{ if } \mr{char}(k)=p>0.$
  \end{cor}

 \begin{proof}
 Recall the Brauer-Picard exact sequence:
\begin{equation}\label{elw13}
 0 \to \mr{Pic}(X) \to \mr{Pic}(X_{\ov{K}})^G \xrightarrow{\mr{br}_X} \mr{Br}(K) \to \mr{Br}(X).
\end{equation}
 Note that, in this case $\mr{Br}(K)=0$. Hence, every $G$-invariant invertible sheaf on $X_{\ov{K}}$ descends to an invertible sheaf on $X$.
 By Theorem \ref{elw9}, this implies  \[\mr{ind}_{\mr{lin}}(X)=\mr{gcd}\{\chi(\mc{L}_i(n))|i=1,...,r \mbox{ and } n \in \mb{Z}\} = 1.\]
 The corollary then follows directly from Lemma \ref{elw01}.
\end{proof}

The corollary gives numerous examples of smooth, projective varieties with index $1$.  
\begin{exa}\label{elw8}
  Suppose $K$ is the quotient field of a Henselian discrete valuation ring $R$ with algebraically closed residue field $k$.
Let $X$ be a smooth, projective variety with $\deg(H_{\ov{K}})>2$, $H^1(\mathcal{O}_X)=0$, $\mr{Pic}(X_{\ov{K}})$ is of rank $2$ and there exists an invertible sheaf $\mc{L}_0$ of degree coprime to $\deg(H_{\ov{K}})$
(for example, see Remark \ref{elw11} below).
Theorem \ref{elw9} implies that every invertible sheaf on $X_{\ov{K}}$ is $G$-invariant and Corollary \ref{elw6} implies that
\[\mr{ind}(X)=\mr{ind}_{\mr{lin}}(X)=1.\]
Indeed, we simply need to check that the two conditions in Theorem \ref{elw9} are satisfied.
Let $\mc{L}_1$ and $\mc{L}_2:=H_{\ov{K}}$ be the generators of $\mr{Pic}(X_{\ov{K}})$. 
Since $\mc{L}_0$ is a linear combination of $\mc{L}_1$ and $\mc{L}_2$ and $\mr{gcd}(\deg(\mc{L}_0),\deg(H_{\ov{K}}))=1$, we have $\mr{gcd}(\deg(\mc{L}_1),\deg(H_{\ov{K}}))=1$.
In other words, the ideal $(\deg(\mc{L}_1), \deg(\mc{L}_2))=1$ i.e., condition $(1)$ of Theorem \ref{elw9} is satisfied.
Let \[A = \left( \begin{array}{cc}
a_0 & a_1  \\
0 & 1  \end{array} \right)\]
be a matrix with integral entries. Note that, for any integer $b>0$
\[A^b = \left( \begin{array}{cc}
a_0^b & a_1(a_0^{b-1}+a_0^{b-2}+...+1)  \\
0 & 1  \end{array} \right)\]
Then, $A^b=\mr{Id}$ if and only if $a_0^b-1=0=a_1(a_0^{b-1}+a_0^{b-2}+...+1)$.
If $a_1 \not=0$ then $a_0=-1$.
 Since $\deg(\mc{L}_2)>2$ is coprime to $\deg(\mc{L}_1)$, we have $2\deg(\mc{L}_1) \not= a_1\deg(\mc{L}_2)$ for any integer $a_1$.
 Thus condition $(2)$ of Theorem \ref{elw9} is satisfied. Hence, by Theorem \ref{elw9} and Corollary \ref{elw6}, we conclude that $\mc{L}_1$ is $G$-invariant and
$\mr{ind}(X)=\mr{ind}_{\mr{lin}}(X)=1.$
\end{exa}

\begin{rem}\label{elw11}
Fix coordinates $X_0, X_1, X_2, X_3$ on $\mb{P}^3_{\mb{Q}}$.
Take any $d \ge 4$ and $F_1, F_2$ two homogeneous polynomials of degree $d$ in variables $X_i$ and coefficients in $\mb{Q}$. It is easy to check that for 
general $F_1, F_2$, the surface defined by $F:=F_1X_1+F_2X_2$ is smooth and $\mr{rk}(\mr{Pic}(X)) = 2$ (the hyperplane section and the line 
defined by $X_1, X_2$ generate $\mr{Pic}(X)$). For any prime $p$, take $K=\mb{Q}_p^{ur}$, the maximal unramified extension of $\mb{Q}_p$.
Note that $K$ is a $C_1$-field, hence $\mr{Br}(K)=0$. Let $X$ be the surface in $\mb{P}^3_K$ defined by $F$. 
By \cite[Ex. III.$5.5$]{R3}, we have $H^1(\mo_X)=0$.
Then Example \ref{elw8} implies that $\mr{ind}_{\mr{lin}}(X)=1=\mr{ind}(X)$.
One can similarly construct numerous examples of surfaces in $\mb{P}^3_K$ of any degree, satisfying the conditions of Theorem \ref{elw9},
arising from the theory of Noether-Lefschetz locus (see \cite{v2, v3}), thereby having index $1$.
\end{rem}

We now give some examples in the case $\mr{Br}(K) \not= 0$, in particular $K$ is not a $C_1$-field.
\begin{exa}\label{elw02}
Let $K=\mb{R}$ and $G:=\mr{Gal}(\ov{K}/K)$ the absolute Galois group. Recall, $\mr{Br}(\mb{R})=\mb{Z}/2\mb{Z}$. 
Denote by $\mr{br}_X: \mr{Pic}(X_{\ov{K}})^G \to \mr{Br}(\mb{R})$ as in the Brauer-Picard exact sequence \eqref{elw13}.
\begin{enumerate}
 \item We first consider the case, when $\mr{br}_X$ is the zero map.
Let $X$ be the smooth, projective surface in $\mathbb{P}^3_{\mb{R}}$ 
defined by the equation $X_0^2+X_1^2+X_2^2-X_3^2=0$, where $X_i$ are the coordinates of $\p3$ for $i=0,...,3$. 
Note that, $X_{\mb{C}}:=X \times_{\mb{R}} \Spec(\mb{C})$ contains the two lines $L_1:=Z(X_0-iX_1, X_2-X_3)$ 
and $L_2:=Z(X_0+iX_1, X_2-X_3)$. The element of the absolute Galois group $G$ 
sending $i$ to $-i$ interchanges $L_1$ and $L_2$. Since $\mr{Pic}(X_{\mb{C}})=\mb{Z}^{\oplus 2}$, generated by $L_1$ and $L_2$, we conclude that 
the $G$-invariant subgroup  of $\mr{Pic}(X_{\mb{C}})$ is generated as a $\mb{Z}$-module
by $L_1+L_2$, which is linearly equivalent to the hyperplane section $H_{\mb{C}}:=H \otimes_{\mb{R}} \mb{C}$. 
Therefore, in this case $\mr{Pic}(X_{\mb{C}})^G$ consists of points corresponding to multiples of the invertible sheaf $H_{\mb{C}}$.
Since $H_{\mb{C}}$ comes from $X$, the exactness of \eqref{elw13} implies that $\mr{br}_X$ is the zero map.
It is easy to check that $\mr{ind}_{\mr{lin}}(X)=1$ (Euler characteristic of $\mo_X$ is $1$). 
Note that, $\mr{ind}(X)=1$ as $X$ contains $\mb{R}$-rational points.

\item We now consider the case, when $\mr{br}_X$ is non-trivial (equivalently surjective). 
Let $X$ be the $\mb{R}$-plane conic defined by $X_0^2+X_1^2+X_2^2=0$, where $X_i$ are the coordinates of $\mb{P}^2_{\mb{R}}$, for $i=0,1,2$.
In this case, $h^1(\mo_X)=0$, i.e., $\mr{Pic}^0(X)=0$. This implies, there exists an unique invertible sheaf $\mc{L}_{\ov{K}}$ on $X_{\ov{K}}$
of degree $1$, hence it is $G$-invariant. If $\mr{br}_X(\mc{L}_{\ov{K}})=0$, then by the exact sequence \eqref{elw13}
there exists an invertible sheaf $\mc{L}$ on $X$ such that $\mc{L}_{\ov{K}} \cong \mc{L} \otimes_K \ov{K}$.
Since $\deg(\mc{L})=1$, the Riemann-Roch theorem would then imply that $X$ contains a rational point, which gives us a contradiction.
Hence, $\mr{br}_X$ is non-trivial.
Using \cite[($1.3$)]{koll}, we have $\mr{ind}_{\mr{lin}}(X)$ is the gcd of $\mr{ind}(X)$ and $1-\rho_a(X)$. 
Since $\mr{ind}(X)=2$, we have $\mr{ind}_{\mr{lin}}(X)=\mr{ind}(X)=2$ (observe $\rho_a(X)=1$).
\end{enumerate}
\end{exa}


\begin{thebibliography}{ELW15}

\bibitem[BFM75]{baum}
P.~Baum, W.~Fulton, and R.~MacPherson.
\newblock Riemann-roch for singular varieties.
\newblock {\em Publications Math{\'e}matiques de l'IH{\'E}S}, 45:101--145,
  1975.

\bibitem[ELW15]{esn2015}
H.~Esnault, M.~Levine, and O.~Wittenberg.
\newblock Index of varieties over {H}enselian fields and {E}uler characteristic
  of coherent sheaves.
\newblock {\em Journal of Algebraic Geometry}, 24(4):693--718, 2015.

\bibitem[Gro66]{ega43}
A.~Grothendieck.
\newblock {\'E}l{\'e}ments de g{\'e}om{\'e}trie alg{\'e}brique (r{\'e}dig{\'e}s
  avec la collaboration de {J}ean {D}ieudonn{\'e}): {IV}. {\'e}tude locale des
  sch{\'e}mas et des morphismes de sch{\'e}mas, troisi{\`e}me partie.
\newblock {\em Publications Math{\'e}matiques de l'IH{\'E}S}, 28:5--255, 1966.

\bibitem[Har10]{R3}
R.~Hartshorne.
\newblock {\em Deformation Theory}.
\newblock Graduate text in Mathematics. Springer-Verlag, 2010.

\bibitem[HL10]{huy}
D.~Huybrechts and M.~Lehn.
\newblock {\em The geometry of moduli spaces of sheaves}.
\newblock Springer, 2010.

\bibitem[Kau16]{ind}
I.~Kaur.
\newblock {\em The ${C_1}$ conjecture for the moduli space of stable vector
  bundles with fixed determinant on a smooth projective curve}.
\newblock Ph. d. thesis, Freie University Berlin, 2016.

\bibitem[Kau18]{ind2}
I.~Kaur.
\newblock A pathological case of the {$C_1$} conjecture in mixed
  characteristic.
\newblock {\em Mathematical Proceedings of the Cambridge Philosophical
  society}, 2018.

\bibitem[Kol13]{koll}
J.~Koll{\'a}r.
\newblock Esnault-{L}evine-{W}ittenberg indices.
\newblock {\em arXiv preprint arXiv:1312.3923}, 2013.

\bibitem[Lan52]{lang1}
S.~Lang.
\newblock On quasi algebraic closure.
\newblock {\em Annals of Mathematics}, pages 373--390, 1952.

\bibitem[Ser13]{ser}
J.~P. Serre.
\newblock {\em Local fields}, volume~67.
\newblock Springer Science \& Business Media, 2013.

\bibitem[Voi88]{v2}
C.~Voisin.
\newblock Une pr\'{e}cision concernant le th\'{e}or\`{e}me de {N}oether.
\newblock {\em Math. Ann.}, 280(4):605--611, 1988.

\bibitem[Voi89]{v3}
C.~Voisin.
\newblock Composantes de petite codimension du lieu de {N}oether-{L}efschetz.
\newblock {\em Comm. Math. Helve.}, 64(4):515--526, 1989.

\end{thebibliography}
\end{document}